\documentclass[11pt, reqno]{amsart}
\usepackage{calc}
\usepackage{cite}
\usepackage{amssymb,amsmath,latexsym,amsthm,enumerate,amsfonts,graphicx,times}
\usepackage[mathscr]{euscript}
\usepackage{amssymb}
\usepackage{amsmath}
\usepackage{color}

\def\dj{d\kern-0.4em\char"16\kern-0.1em}
\def\Dj{\mbox{\raise0.3ex\hbox{-}\kern-0.4em D}}

\def\be{\begin{equation}}
\def\ee{\end{equation}}
\def\bena{\begin{eqnarray*}}
\def\ena{\end{eqnarray*}}

\def\mR{\mathbb{R}}

\def\t{\tau}
\def\s{\sigma}

\def\suml{\sum\limits}

\def\dss{\displaystyle}

\newcommand{\WF}{\operatorname{WF}}

 \def\D{\mathcal{D}}
 \def\E{\mathcal{E}}
 \def\Rd{\mathbf{R}^d}
 \def\Z{\mathbf{Z}_+}

\def\N{\mathbf{N}}

\def\im{\rm{Im\, }}
\def\re{\rm{Re\, }}
\newcommand{\supp}{\operatorname{supp}}
\numberwithin{equation}{section}

\newtheorem{te}{Theorem}[section]
\newtheorem{lema}{Lemma}[section]
\newtheorem{prop}{Proposition}[section]
\newtheorem{cor}{Corollary}[section]

\theoremstyle{definition}

\newtheorem{de}{Definition}[section]
\theoremstyle{remark}
\newtheorem{rem}{Remark}[section]

\title{\textbf{Boundary values in ultradistribution spaces related to  extended Gevrey regularity}}

\author{Stevan Pilipovi\' c}

\address{Department of Mathematics and Informatics,
University of Novi Sad, Novi Sad, Serbia}

\email{stevan.pilipovic@dmi.uns.ac.rs}

\author{Nenad Teofanov}

\address{Department of Mathematics and Informatics,
University of Novi Sad, Novi Sad, Serbia}

\email{nenad.teofanov@dmi.uns.ac.rs}

\author{Filip Tomi\'c}

\address{Faculty of Technical Sciences,
University of Novi Sad, Novi Sad, Serbia}

\email{filip.tomic@uns.ac.rs}



\begin{document}
\begin{abstract}
Following the well- known theory of Beurling and Roumieu ul\-tra\-di\-stri\-butions,
we investigate new spaces of ultradistributions as dual spaces of test functions which correspond to associated functions of logarithmic-type growth at infinity.
In the given framework we prove that boundary values of analytic functions with the corresponding logarithmic growth rate towards the real domain are ultradistributions. The essential condition for
that purpose, condition $(M.2)$ in the classical ultradistribution theory, is replaced by the new one, $\widetilde{(M.2)}$.
For that reason, new techniques were performed in the proofs. As an application,
we discuss the corresponding wave front sets.
\end{abstract}

\maketitle

\par

\section{Introduction}

In this paper we describe certain intermediate spaces between the space of \linebreak Schwartz distributions and any space of Gevrey ultradistributions
as boundary values of analytic functions. More precisely, we continue to investigate a new class of ultradifferentiable functions and their duals (\cite{PTT-01,
PTT-02, PTT-03, PTT-04}) following Komatsu's approach \cite{Komatsuultra1, KomatsuNotes}. We refer to \cite{BMT} and the references therein for another, equally interesting approach,

The derivatives of such ultradifferentiable functions  are controlled by the two-parameter sequences of the form  $M^{\t,\s}_p=p^{\t p^{\s}}$, $p\in \N$, $\t>0$, $\s>1$. For that reason we call them extended Gevrey functions. It turned out that such functions can be used in the study of a class of strictly hyperbolic equations and systems. In particular, the extended Gevrey class associated to the sequence $M^{1,2}_p=p^{p^2}$
is used in the analysis of the regularity of the corresponding Cauchy problem in \cite{CL}.
It captures the regularity of the coefficients in the space variable
(with low regularity in time), so that the corresponding Cauchy problem is well posed
in appropriate solution spaces.

Actually, the change of  the growth rate of  sequence $M_{p}^{\tau,\s}$
implied a change in the growth of the expression $h^p$ in classical definition (see \cite{Komatsuultra1}). Hence, instead of that expression we use
 $h^{p^\sigma}$ which essentially changes the corresponding proofs in the analysis of new ultradistribution spaces. We especially emphasize
the role of the Lambert $W$ function that  appears  in the theory of new  ultradistribution spaces for  the first time. This is the essential contribution of our approach. The properties of new ultradistribution spaces
described in terms of the Lambert function and its asymptotic properties
show that our approach is naturally included in the general theory of ultradistributions positioning the new spaces, let us call them extended Gevrey ultradistributions, between classical distributions and Komatsu type ultradistributions.

Distributions as boundary values of analytic functions are investigated in many papers, see \cite{PilipovicKnjiga} for the historical background and the relevant references therein. We point out a nice survey for distribution and ultradistribution boundary  values given in the book  \cite{H}. The essence of the existence of  a boundary value  is the  determination of the growth condition under which  an analytic function $F(x+iy)$, observed on a certain tube domain with respect to $y$, defines an (ultra)distribution as  $y$ tends to $0$.  The classical result can be roughly interpreted as follows:  if $F(x+iy)\leq C |y|^{-M}$ for some $C,M>0$ then $F(x+i 0)$ is in
the Schwartz space $\D ' (U)$ in a neighborhood $U$ of $x$. (see Theorem 3.1.15 in \cite{H}).
For Gevrey ultradistributions, sub-exponential growth rate of analytic function $F$
of the form  $|F(x+iy)|\leq C e^{k|y|^{-1/(t-1)}}$ for some $C,k>0$ and $ t>1$
implies the  boundary value result.  The function in the exponent  precisely describes the asymptotic behavior of the
associated function to the Gevrey sequence $ p!^{t},$ $ p \in \N$, cf. \cite{KomatsuNotes, Rodino}.
In general, such representations are provided if test functions admit \emph{almost analytic extensions}  in the non-quasianalytic case related to
Komatsu's condition $(M.2)$ (see \cite{Pilipovic-01}).

Different  results concerning boundary values in the spaces of ultradistributions can be found in \cite{KomatsuNotes, Komatsuultra1, PilipovicKnjiga, Pilipovic-01, Pilipovic-02}. Even now this topic for ultradistribution spaces is interesting (cf. \cite{DV, DPV, FGG, VV}). Especially, we have to mention \cite{Rainer}. At
the end of this  introduction we will briefly comment on the approach in this paper and our approach.

\par

Extended Gevrey classes $\E_{\t,\s}(U)$ and $\D_{\t,\s}(U)$, $\t>0$, $\s>1$, are introduced and investigated in \cite{PTT-01}-\cite{PTT-04},  \cite{TT0, TT}. The derivatives of functions in such classes are controlled by sequences of the form $M_p^{\t,\s}=p^{\t p^{\s}}$, $p\in \N$. Although such sequences do not satisfy Komatsu's condition $(M.2)$, the corresponding spaces consist of ultradifferentiable  functions, that is, it is possible to construct differential operators of infinite order and prove their
continuity properties on the test and dual spaces.

Our main intention in this paper is to establish the sufficient condition when the elements of  dual spaces can be represented as boundary values of analytic functions.
We follow the classical approach to boundary values given in \cite{PilipovicKnjiga} and
carry out necessary modifications in order to use it in the analysis of spaces developed in  \cite{PTT-01}-\cite{PTT-04}.
Here, for such spaces,  a plenty of non-trivial constructions are performed. In particular,
we analyze the corresponding associated functions as a main tool in our investigations.

Moreover, we apply these results in the description of related wave front sets.
The wave front set $\WF_{\t,\s}(u)$, $\t>0$, $\s>1$, of a Schwarz distribution $u$ is analyzed in \cite{ PTT-02, PTT-03, PTT-04, TT0, TT}.
In particular, it is proved that they are related to the classes $\E_{\t,\s}(U)$. We extend the definition of $\WF_{\t,\s}(u)$ to a larger space of ultradistributions by using their boundary value representations. This allows us to describe intersections and unions of $\WF_{\t,\s}(u)$ (with respect to $\t$) by using
specific functions with logarithmic type behaviour.

Let us comment on another, very interesting, concept of construction of a large class of ultradistribution spaces.
In \cite{Rainer, KMR2,RS} and several other papers the authors consider sequences of the form $k!M_k$,
where they presume a fair number of conditions on $M_k$ and discuss in details  their relations.
For example, consequences to the composition of ultradifferentiable functions determined by different classes of such sequences  are discussed.
Moreover, they consider weighted matrices, that is a family of sequences of the form $k!M_k^\lambda$, $k\in\mathbb N,$ $\lambda\in \Lambda$
(partially ordered and directed set) and make the unions, again considering various properties such as compositions and boundary values.
Their analysis follows the approach of \cite{BMT, MV}.
In essence, an old question of ultradistribution theory was the analysis of unions and intersections of ultradifferentiable
function spaces. This is very well elaborated in quoted papers.
The main reason why our classes are not covered by the quoted papers is the factor $h^{|\alpha|^{\s}}$, $\s>1$, in the seminorm \eqref{Norma}. 
For that reason our conditions on the weight sequence ($\widetilde{(M.2)'}$ and $\widetilde{(M.2)}$ below) differ from the corresponding ones in the quoted papers. As we already explained, our growth rate is not just another point of view since the basic facts used in the proofs are related to a new investigations involved by the Lambert $W$ function.
Actually, the precise estimates of our paper can be used for the further extensions in matrix approach since the original idea for our approach is quite different and based on the relation between $[n^s]!$  and $n!^s$ in the estimate of derivatives ($[n^s]$ means integer value not exceeding $n^s$, $s\in(0,1)$, cf. \cite{PTT-01,PTT-02}).

\par

The paper is organized as follows: We end the introduction with some notation.
In Section \ref{secSpaces} we introduce the necessary background on the spaces of extended Gevrey functions and their duals, spaces of ultradistributions.
Our main result Theorem \ref{GlavnaTeorema} is given in Section \ref{secGlavna}.
Wave front sets in the framework of our theory are discussed in  Section \ref{secTalasniFront}.
Finally, in Appendix we prove a technical result concerning the associated functions $T_{\t,\s,h}(k)$
and recall the basic continuity properties of ultradifferentiable operators on extended Gevrey classes,
in a certain sense analogous to  stability under the ultradifferentiable operators in the classical theory.

\subsection{Notation}\label{secNotacija}
We denote by ${\bf N}$, $\Z$, ${\bf R}$, ${\bf C}$ the sets of nonnegative
integers, positive integers, real numbers and complex numbers, respectively.
For a multi-index
$\alpha=(\alpha_1,\dots,\alpha_d)\in {\bf N}^d$, we write
$\partial^{\alpha}=\partial^{\alpha_1}\dots\partial^{\alpha_d}$, $\dss D^{\alpha}= (-i )^{|\alpha|}\partial^\alpha$, and
$|\alpha|=|\alpha_1|+\dots +|\alpha_d|$. The open ball $B_r(x_0)$ has radius $r>0$ and center at $x_0\in\Rd$;  $\dss \partial_{\overline{z}}=( \partial_{\overline{z}_1},..., \partial_{\overline{z}_n})$ where $ \partial_{\overline{z}_j} =\frac{1}{2}(\partial_{x_j}+i\partial_{y_j} )$,
$ j = 1,\dots, d$, $z=x+iy\in {\bf C} ^d$. By Hartogs's theorem, $f(z), z\in\Omega,$ $\Omega$ is open in ${\bf C} ^d$, is analytic if it is analytic with respect to every coordinate variable $z_i.$

\par

Throughout the paper we always assume  $\t>0$ and $\s>1$.

\section{Test Spaces and Duals}\label{secSpaces}

We are interested in $M_p^{\tau,\s}$, $p\in \N$, sequences of positive numbers such that   conditions $(M.1)$ and $(M.3)$ of \cite{Komatsuultra1} hold,
and instead of $(M.2)'$ and $(M.2)$ of \cite{Komatsuultra1}, for some $C>1,$ we have

$\widetilde{(M.2)}$ $M_{p+q}^{\t,\s}\leq C^{p^{\s}+q^{\s}+1}M_p^{2^{\s-1}\t ,\s}M_q^{2^{\s-1}\t,\s}$, $p,q\in \N$,

$\widetilde{(M.2)'}$ $M_{p+1}^{\t,\s}\leq C^{p^{\s}+1}M_p^{\t,\s}$, $p\in \N$.

In the sequel we consider the sequence $M_p^{\tau,\s}=p^{\t p^{\s}}$, $p\in \N$, which fulfills the above mentioned conditions
(see \cite[Lemma 2.2.]{PTT-01}). This particular choice slightly simplifies our exposition.
Clearly, by choosing $\s=1$ and $\t>1$ we recover the well known Gevrey sequence $p!^\t$.

Recall \cite{PTT-04}, the associated function related to the sequence $p^{\t p^{\s}}$, is defined by
\be \label{asociranaProduzena}
\dss T_{\t,\s,h}(k)=\sup_{p\in \N}\ln\frac{h^{p^{\s}}k^{p}}{p^{\t p^{\s}}}, \;\;\; k>0.
\ee
For $h,\s=1$ and $\t>1$, $T_{\t, 1,1}(k)$ is the associated function to the Gevrey sequence $p!^{\t}$.

In the next lemma we derive the precise asymptotic behaviour of the function $T_{\t,\s,h}$ associated with the sequence $p^{\t p^{\s}}$.
This in turn highlights  the essential difference between
$T_{\t,\s,h}$ and the associated functions determined by Gevrey type sequences.

\par

We first introduce some notation. The Lambert $W$ function is defined as the inverse function of $z e^{z}$, $z\in {\bf C}$, wherefrom 
\be
\label{osobinaLambert}
\dss x=W(x)e^{W(x)}, \quad x\geq 0.
\ee
We denote its principal (real) branch by $W(x)$, $x\geq 0$ (see \cite{LambF, HoHa}).
It is a continuous,  increasing and concave function on $[0,\infty)$, $W(0)=0$, $W(e)=1$, and $W(x)>0$, $x>0$.
It can be shown that $W$ can be represented in the form of the absolutely convergent series
$$
W(x)=\ln x-\ln (\ln x)+\sum_{k=0}^{\infty}\sum_{m=1}^{\infty}c_{km}\frac{(\ln(\ln x))^m}{(\ln x)^{k+m}},\quad x\geq x_0>e,
$$
with suitable constants $c_{km}$ and  $x_0 $. Thus  the following  estimates hold:
\be
\label{sharpestimateLambert}
\ln x -\ln(\ln x)\leq W(x)\leq \ln x-\frac{1}{2}\ln (\ln x), \quad x\geq e,
\ee
with the equality in \eqref{sharpestimateLambert} if and only if $x=e$.

\par

For given $\s>1$, $\t,h>0$, let
$$\dss {\mathfrak R}(h,\cdot):=h^{-\frac{\s-1}{\t}}e^{\frac{\s-1}{\s}}\frac{\s-1}{\t \s}\ln k
= h^{-\frac{\s}{\t \s '}}e^{\frac{1}{\s'}}\frac{1}{\t \s'}\ln k,\quad k>e, $$
where
$$
\frac{1}{\s} + \frac{1}{\s '} = 1, \;\;\; \text{i.e.} \;\;\; \s' = \frac{\s}{\s -1}.
$$

\par

\begin{lema}
\label{OcenazaAsociranu}
Let $h>0$, and let $T_{\t,\s,h}$ be given by \eqref{asociranaProduzena}. Then there
exists constants
$B_1,B_2,b_1,b_2>0$
such that
$$
B_1\, k^{b_1\Big(\frac{\ln k}{\ln (\ln  k) }\Big)^{\frac{1}{\s-1}}}\leq
\exp\{T_{\t,\s,h}(k)\}\leq
B_2 \,k^{ b_2 \Big( \frac{\ln k}{\ln (\ln  k) }\Big)^{\frac{1}{\s-1}}}, \quad k>e.
$$
More precisely, if
$$ \dss c_1=\Big(\frac{\s-1}{\t \s}\Big)^{\frac{1}{\s-1}},\;\;\;
\text{and} \;\;\; c_2=h^{-\frac{\s-1}{\t}}e^{\frac{\s-1}{\s}}\frac{\s-1}{\t \s},
$$
then, there exist constants $A_1, A_2>0$ such that
$$
A_1\, k^{\frac{1}{2}\frac{\s-1}{ \s}c_1\Big( \frac{\ln k}{\ln ( c_2 \ln  k) }\Big)^{\frac{1}{\s-1}}}\leq
\exp\{T_{\t,\s,h}(k)\}\leq
A_2 \,k^{ c_1 \Big( \frac{\ln k}{\ln (c_2 \ln  k) }\Big)^{\frac{1}{\s-1}}}, \quad k>e.$$
\end{lema}

\begin{proof}
Lemma \ref{OcenazaAsociranu} can be proved by following the arguments used in the proof of \cite[Theorem 2.1]{PTT-04}.
There it is shown that for given
$h>0$, $\t>0$ and $\s>1$  the following inequalities hold:
\begin{multline}
\label{nejednakostzaTeoremu1}
\tilde{A}_{\t,\s,h}\exp \Big\{ \Big (
\frac{{\ln} k }{(2^{\s-1 }\t W ({\mathfrak R}(h,k)) )^{\frac{1}{\s} } \s' } \, \,
\Big )^{\s'} \Big\}
\leq e^{T_{\t,\s,h}(k)}\\ \leq
 A_{\t,\s,h}\exp\Big\{ \Big (
 \frac{{\ln}k }{(\t  \s' W({\mathfrak R}(h,k)))^{\frac{1}{\s}}}\,  \,
 \Big )^{\s'}
 \Big\}, \quad k>e,
\end{multline} for some $A_{\t,\s,h}, \tilde{A}_{\t,\s,h}>0$.
Moreover, in the view of \eqref{sharpestimateLambert}, it follows that
\begin{equation}
\label{asimptotskaocena}
 {W^{-\frac{\s'}{\s}}({\mathfrak R}(h,k))}\,({\ln}k)^{\s'}
\asymp
\Big(  \frac{\ln k}{\ln (C_h \ln k)}\Big)^{\frac{\s'}{\s}}\ln k,\quad k\to \infty,
\end{equation} 
with $\dss C_h:= h^{-\frac{\s-1}{\t}}e^{\frac{\s-1}{\s}}\frac{\s-1}{\t \s} $
$ = h^{-\frac{\s}{\t \s'}}e^{\frac{1}{\s'}} (\t \s')^{-1}.$ 

Details are left for the reader.
\end{proof}

We define (following the classical approach \cite{Komatsuultra1}):
\be \label{Tstar}
\dss T^{*}_{\t,\s,h}(k)=\sup_{p\in \N}\ln\frac{h^{p^{\s}}k^{p}}{p^{p (\t p^{\s-1}-1)}}, \;\;\; k>0.
\ee
It turns out that $T^{*}_{\t,\s,h}(k)$ enjoys the same asymptotic behaviour as  $T_{\t,\s,h}$, cf. Lemma \ref{NejednakostLema2} a) in the Appendix.
This is another difference between our approach and the  classical ultradistribution theory, where $T^*$ plays an important role.

\par

Next we recall  the definition of spaces $\E_{\t,\s}(U)$ and $\D_{\t,\s}(U)$, where $U$ is an open set in $\Rd$ (\cite{PTT-01}).

Let $K\subset \subset \Rd$ be a regular compact set. Then, ${\E}_{\t, {\s},h}(K)$
is
the Banach space of  functions $\phi \in  C^{\infty}(K)$ such that
\begin{equation} \label{Norma}
\| \phi \|_{{\E}_{\t, {\s},h}(K)}=\sup_{\alpha \in \N^d}\sup_{x\in K}
\frac{|\partial^{\alpha} \phi (x)|}{h^{|\alpha|^{\s}}  |\alpha|^{\t |\alpha|^{\s}} }<\infty.
\end{equation}
We have
$$ \displaystyle
{\E}_{\t_1, {\s_1},h_1}(K)\hookrightarrow {\E}_{\t_2,
{\s_2},h_2}(K), \;\;\;
0<h_1< h_2, \; 0<\t_1<\t_2, \; 1<\s_1< \s_2,
$$
where $\hookrightarrow$ denotes the strict and dense inclusion.

The set of functions from ${\E}_{\t,
\s,h}(K)$ supported by $K$ is denoted by  ${\D}^K_{\t, \s,h}$ .
Next,
\begin{equation}
\label{NewClassesInd} {\E}_{\{\t,
\s\}}(U)=\varprojlim_{K\subset\subset U}\varinjlim_{h\to
\infty}{\E}_{\t, {\s},h}(K),
\end{equation}
\begin{equation}
\label{NewClassesProj} {\E}_{(\t,
\s)}(U)=\varprojlim_{K\subset\subset U}\varprojlim_{h\to 0}{\E}_{\t,
{\s},h}(K),
\end{equation}
\begin{equation}
\label{NewClassesInd2} {\D}_{\{\t,
\s\}}(U)=\varinjlim_{K\subset\subset U} {\D}^K_{\{\t, \s\}}
=\varinjlim_{K\subset\subset U} (\varinjlim_{h\to\infty}{\D}^K_{\t,
\s,h})\,,
\end{equation}
\begin{equation}
\label{NewClassesProj2} {\D}_{(\t,
\s)}(U)=\varinjlim_{K\subset\subset U} {\D}^K_{(\t, \s)}
=\varinjlim_{K\subset\subset U} (\varprojlim_{h\to 0}{\D}^K_{\t,
\s,h}).
\end{equation}
Spaces in \eqref{NewClassesInd} and \eqref{NewClassesInd2} are called \emph{Roumieu} type spaces while \eqref{NewClassesProj} and \eqref{NewClassesProj2} are \emph{Beurling} type spaces.
Note that all the spaces of ultradifferentiable functions defined by Gevrey type sequences are contained in the corresponding spaces defined above.

For the corresponding spaces of ultradistributions we have:
\begin{equation*}
\label{NewClassesIndDist}
{\D'}_{\{\t,
\s\}}(U)=\varprojlim_{K\subset\subset U}\varprojlim_{h\to 0} (\D^{K} _{\t, {\s},h})',
\quad {\D'}_{(\t, \s)}(U)=\varprojlim_{K\subset\subset U}\varinjlim_{h\to \infty}
(\D ^{K}_{\t,{\s},h})'.
\end{equation*}
Topological properties of all those spaces  are the same as in the case of Beurling and Roumieu type spaces given in \cite{Komatsuultra1}.

We will use abbreviated notation $ \t,\s $ for $\{\t,\s\}$ or $(\t,\s)$.
Clearly,
$$
\dss \D' (U)\hookrightarrow \D'_{\t,\s} (U)\hookrightarrow \varprojlim_{t\to 1} \D'_{t}(U),
$$
where $\D'_{t}(U)=\D'_{t,1}(U)$ denotes the space of Gevrey ultradistributions with index $t>1$.
More precisely, if (for $\s>1$) we put
\begin{equation*}
{\D}^{(\s)}(U)=\varprojlim_{\t\to 0} {\D}_{\t,\s}(U),\quad
\text{and}
\quad {\D}^{\{\s\}}(U) =\varinjlim_{\t\to \infty} {\D}_{\t,\s}(U),
\end{equation*}
then
$$\dss \D '(U)\hookrightarrow {\D}^{'\{\s\}}(U)\hookrightarrow {\D}^{'(\s)}(U)\hookrightarrow \varprojlim_{t\to 1} \D'_{t}(U),$$ where $ {\D}^{'(\s)}(U)$ and   ${\D}^{'\{\s\}}(U)$ are dual spaces of ${\D}^{(\s)}(U)$ and ${\D}^{\{\s\}}(U)$,
respectively.

Thus we are dealing with intermediate spaces
between the space of Schwartz distributions and spaces of Gevrey ultradistributions.
In the next section we show the boundary value result in the given framework.
This, however asks for  the use of new techniques.

\section{Main result}\label{secGlavna}

The condition $(M.2)$ (also known as the stability under the ultradifferentiable operators), essential for the boundary value theorems in the framework of ultradistribution spaces \cite{Pilipovic-01, Komatsuultra1}, is in our approach replaced by the condition $\widetilde{(M.2)}$.
 We note that in \cite{Rainer} a more general condition than  $(M.2)$ is considered.
In the case of the sequence $M_p^{\tau,\s}=p^{\t p^{\s}}$, $p\in \N$, the asymptotic behaviour given in Lemma \ref{OcenazaAsociranu}
is essentially used to prove our main result  as follows.

\begin{te}
\label{GlavnaTeorema}
Let $\s>1$, $U$ be an open set in $\Rd$, $\Gamma$ an open cone in $\Rd$ and  $\gamma >0$. Assume that $F(z)$, $ z \in Z $  is an analytic function, where
$$ Z=\{z\in {\mathbf C}^d\,|\, {\re} {z}\in U\,, {\im} z\in \Gamma, |{\im} z|<\gamma\},$$
and such that
\be
\label{uslovVelikoFPosledica}
|F(z)|\leq A |y|^{ - H \Big( \frac{\ln (1/|y|)}{\ln ( \ln (1/|y|)) }\Big)^{\frac{1}{\s-1}}},\quad z=x+iy\in Z,\nonumber
\ee
for some $A,H>0$ (resp. for every $H>0$ there exists $A>0$). Then
\be
\label{LimitGlavnaTeorema}
\dss F(x+iy)\to F(x+i 0), \quad y\to 0,\, y\in \Gamma,
\ee
in $ {\D}^{'(\s)}(U)$ (resp. $ {\D}^{'\{\s\}}(U)$).

More precisely, if
\be
\label{uslovVelikoF}
|F(z)|\leq A \exp\{T_{(2^{\s}-1)\t,\s,H}(1/|y|)\}\quad z=x+iy\in Z,
\ee for some $A,H>0$ (resp. for every $H>0$ there exists $A>0$) then
  \eqref{LimitGlavnaTeorema} holds in ${\D'}_{(\t/2^{\s-1},
\s)}(U)$ (resp. ${\D'}_{\{\t/2^{\s-1},
\s\}}(U)$).
\end{te}

\begin{proof}
Let $K\subset\subset U$ and $\varphi\in \D^K_{\t/2^{\s-1},\s}$. Moreover, let $\kappa\in \D_{\t/2^{\s-1},\s} (\Rd)$ be such that $\supp\kappa\subseteq \overline{B(0,2)}$, $\kappa=1$ on $B(0,1)$.

In the sequel we denote $\displaystyle m_p=p^{\t ((2p)^{\s-1}-1)}$, $p\in \N$. Clearly, $m_p$ is an increasing sequence and $m_p\to \infty$ as $p\to \infty$.

Fix $h>0$, and let
\be
\label{KapaFunkcije}
\kappa_\alpha (y)=\kappa (4 h m_{|\alpha|} y),\quad \alpha\in \N^d,\,y\in \Rd.\nonumber
\ee
 Note that
\be
\label{NosacKappa}
\supp \kappa_\alpha\subseteq \{y\in \Rd\,|\,|y|\leq 1/(2 h m_{|\alpha|} )\},
\ee and for $j=1,\dots, d$,
\be
\label{NosacKappaIzvod}
\supp \partial_{y_j}\kappa_\alpha\subseteq \{y\in \Rd\,|\, 1/(4 h m_{|\alpha|})\leq |y| \leq  1/(2 h m_{|\alpha|}) \},\quad \alpha\in \N^d.
\ee

Let
\be
\label{AlmostAnalyticExtension}
\Phi (z)=\sum_{\alpha\in \N^d}\frac{\partial^{\alpha}\varphi (x)}{|\alpha|^{\t |\alpha|}}(i y)^{\alpha}\kappa_\alpha (y),\quad z=x+iy\in {\mathbf C}^d.
\ee
Clearly, $\Phi$ is a smooth function in $\mR^{2d}$ and $\Phi(x)=\varphi (x)$ for $x\in K$.

Fix $Y=(Y_1,\dots, Y_d)\in \Gamma$, $Y\not=0$, $|Y|<\gamma$,  and set
\be
\label{ZY}
\dss Z_Y=\{x+i t Y\,|\, x\in K,\, t\in (0,1]\}.
\ee

In order to use Stoke's formula (see \cite{Pilipovic-01}) we need
to estimate $\Phi$ and its derivatives on $Z_Y$. To that end we had to adjust the standard technique in a nontrivial manner.

Let us show that there exists $A_h>0$ such that
\be
\label{prvaocenaPhi}
|\Phi(z)|\leq A_h \|\varphi\|_{\E_{\t/2^{\s-1},\s,h}},\quad h>0,\,\,  z\in Z_Y.
\ee  Note that \eqref{NosacKappa} implies
$$ |tY|^{|\alpha|}|\kappa_\alpha (tY)|\leq \frac{1}{(2 h m_{|\alpha|})^{|\alpha|}}=\frac{|\alpha|^{\t |\alpha|}}{(2h)^{|\alpha|} |\alpha|^{{2^{\s-1}\t |\alpha|^{\s}}}},\quad t\in (0,1],\, \alpha \in \N^d,$$
and therefore we obtain
\begin{multline}
\label{ocenaPHI}
|\Phi (z)|\leq \sum_{\alpha\in \N^d}\frac{|\partial^{\alpha}\varphi (x)|}{|\alpha|^{\t |\alpha|}}|tY|^{\alpha}|\kappa_\alpha (tY)|\leq \sum_{\alpha\in \N^d}\frac{|\partial^{\alpha}\varphi (x)|}{(2h)^{|\alpha|} |\alpha|^{{2^{\s-1}\t |\alpha|^{\s}}}}\\
 \leq \|\varphi\|_{\E_{\t/2^{\s-1},\s,h}}  \sum_{\alpha\in \N^d}\frac{h^{|\alpha|^{\s}}|\alpha|^{(\t/2^{\s-1}) |\alpha|^{\s}}}{(2h)^{|\alpha|} |\alpha|^{{2^{\s-1}\t |\alpha|^{\s}}}} = A_h \|\varphi\|_{\E_{\t/2^{\s-1},\s,h}},\nonumber
\end{multline} where $\dss A_h =\sum_{\alpha\in \N^d}\frac{h^{|\alpha|^{\s}-|\alpha|}}{2^{|\alpha|}|\alpha|^{\t_0 |\alpha|^{\s}}}<\infty$ for $\t_0=\t (2^{\s-1} - \frac{1}{2^{\s-1}})>0$. Hence \eqref{prvaocenaPhi} follows.

Next we estimate  $\partial_{\overline{z_j}}\Phi(z)$, $j\in \{1,\dots, d\}$,
when $z\in Z_Y$. More precisely, we show that  for a given $ h>0,$ there exists  $B_h>0$ such that
\be
\label{AnaliticExt}
|\partial_{\overline{z_j}}\Phi(z)|\leq B_h \|\varphi\|_{\E_{\t/2^{\s-1},\s,h}} \exp\{-T_{(2^{\s}-1)\t,\s,h}(1/|tY|)\},\quad\, z\in Z_Y.
\ee
By \eqref{NosacKappa} and \eqref{NosacKappaIzvod} it is sufficient to prove \eqref{AnaliticExt} for
\be
\label{NosacVelikoY}
1/(4 h m_{|\alpha|})\leq |tY| \leq  1/(2 h m_{|\alpha|}),\quad   0<t\leq 1,\,\alpha\in \N^d.
\ee

\par

Note that for $z\in Z_Y$ we have
\begin{multline}
\partial_{\overline{z_j}}\Phi(z)=\frac{1}{2}\Big(\sum_{\alpha\in \N^d}\frac{\partial^{\alpha+e_j}\varphi (x)}{|\alpha|^{\t |\alpha|}}i^{|\alpha|}t^{|\alpha|} Y^{\alpha}\kappa_\alpha (tY) \\
+ \sum_{\alpha\in \N^d}\frac{\partial^{\alpha}\varphi (x)}{|\alpha|^{\t |\alpha|}}\alpha_j i^{|\alpha|+1}t^{|\alpha|} Y^{\alpha-e_j}\kappa_\alpha (tY)\\
+ \sum_{\alpha\in \N^d}\frac{\partial^{\alpha}\varphi (x)}{|\alpha|^{\t |\alpha|}}i^{|\alpha|+1}t^{|\alpha|+1}Y^{\alpha} 4 h m_{|\alpha|} \cdot {(\partial_{y_j} \kappa)} ( 4 h m_{|\alpha|}t Y)\Big)
=\frac{1}{2}(S_1+S_2+S_3) (z).\nonumber
\end{multline} 

We will show that there exists a constant $ B_h>0$ such that
\be
\label{KonacanOcena}
\dss \exp\{{T_{(2^{\s}-1)\t,\s,h}}(1/|tY|)\}|S_1 (z)|\leq B_h \|\varphi\|_{\E_{\t/2^{\s-1},\s,h}},\quad z\in Z_Y.\nonumber
\ee
The estimates for $S_2$ and $S_3$ can be obtained in a similar way.

Let $C_h=C \max\{h, h^{2^{\s-1}}\}$ where $C>0$ is the constant from $\widetilde{(M.2)'}$.
Using $(p+1)^{\s}\leq 2^{\s-1} (p^{\s}+1)$, $p\in\N$, we obtain

\begin{multline}
\frac{ h^{|\beta|^{\s}}}{|tY|^{|\beta|}|\beta|^{(2^{\s}-1)\t |\beta|^{\s}}}|S_1(z)|\leq C_h \|\varphi\|_{\E_{\t/2^{\s-1},\s, h}}\\ \Big( \sum\limits_{\substack{\alpha\in \N^d \\ |\alpha|\leq |\beta|}}+ \sum\limits_{\substack{\alpha\in \N^d \\ |\alpha| > |\beta|}}\Big) \frac{h^{|\beta|^{\s}} C^{|\alpha|^{\s}}_h |\alpha|^{(\t/2^{\s-1}) |\alpha|^{\s}}}{|\beta|^{(2^{\s}-1)\t |\beta|^{\s}}|\alpha|^{\t |\alpha|}}|tY|^{|\alpha|-|\beta|}|\kappa_\alpha (tY)|   \\
= C_h\|\varphi\|_{\E_{\t/2^{\s-1},\s, h}}(I_{1,\beta}+I_{2,\beta}),\quad \beta\in \N^d,\,z\in Z_Y.\nonumber
\end{multline} It remains to show that $\sup_{\beta \in \N^d} I_{1,\beta}$ and  $\sup_{\beta \in \N^d}I_{2,\beta}$ are finite.

First we estimate $I_{1,\beta}$. Note that for $|\alpha|\leq |\beta|$, the left-hand side in \eqref{NosacVelikoY} implies
\begin{multline}
\label{OcenaZaI1}
|tY|^{|\alpha|-|\beta|}|\kappa_\alpha (tY)| \leq (4 h m_{|\alpha|})^{|\beta| - |\alpha|}\leq  \frac{(4h)^{|\beta|} m^{|\beta|}_{|\beta|}}{h^{|\alpha|}m^{|\alpha|}_{|\alpha|}}\\
\leq\frac{ (4h)^{|\beta|}|\alpha|^{\t |\alpha|} |\beta|^{2^{\s-1}\t|\beta|^{\s}}}{h^{|\alpha|}|\alpha|^{2^{\s-1}\t |\alpha|^{\s}}},\quad t\in (0,1],\alpha,\beta \in \N^d.
\end{multline}
Again, when $\t_0=\t (2^{\s-1} - \frac{1}{2^{\s-1}})$, by  \eqref{OcenaZaI1} we have
\be
 I_{1,\beta} \leq \frac{(4h)^{|\beta|}h^{|\beta|^{\s}}}{|\beta|^{(2^{\s-1}-1)\t |\beta|^{\s}}}\sum_{\alpha\in \N^d} \frac{C_h^{|\alpha|^{\s}}}{h^{|\alpha|}|\alpha|^{\t_0 |\alpha|^{\s}}}=C'_h \frac{(4h)^{|\beta|}h^{|\beta|^{\s}}}{|\beta|^{(2^{\s-1}-1)\t |\beta|^{\s}}},\,\,\beta \in \N^d.\nonumber
\ee Hence we conclude $\sup_{\beta\in \N^d} I_{1,\beta} \leq C'_h \exp\{T_{(2^{\s-1}-1)\t,\s,h} (4h) \}<\infty$.

To estimate $I_{2, \beta}$ we first note that for $|\alpha|>|\beta|$ the right-hand side in \eqref{NosacVelikoY} implies
\begin{multline}
\label{OcenaZaI2}
 |tY|^{|\alpha-\beta|}|\kappa_\alpha (tY)| \leq (1/(2 h m_{|\alpha|}))^{|\alpha-\beta|}\leq1/({(2h)^{|\alpha-\beta|}m^{|\alpha-\beta|}_{|\alpha-\beta|}})\\
\leq\frac{|\alpha|^{\t |\alpha|}}{(2h)^{|\alpha-\beta|} |\alpha-\beta|^{2^{\s-1}\t |\alpha - \beta|^{\s}}}, \quad t\in (0,1],\,\,\alpha,\beta \in \N^d.
\end{multline}

Set $C''_h=C \max\{C_h, C^{2^{\s-1}}_h\}$. Using $\widetilde{(M.2)}$, \eqref{OcenaZaI2} and \eqref{SimpleInequality}, for $\beta \in \N^d$, we have

\begin{multline}
 I_{2,\beta} \leq  \sum\limits_{\substack{\alpha\in \N^d \\ |\alpha| > |\beta|}} \frac{h^{|\beta|^{\s}} C^{|\alpha|^{\s}}_h |\alpha|^{(\t/2^{\s-1}) |\alpha|^{\s}}}{|\beta|^{(2^{\s}-1)\t |\beta|^{\s}}(2h)^{|\alpha-\beta|} |\alpha-\beta|^{2^{\s-1}\t |\alpha - \beta|^{\s}}}\\
\leq C \sum\limits_{\substack{\alpha\in \N^d \\ |\alpha| > |\beta|}}\frac{h^{|\beta|^{\s}} (C''_h)^{|\alpha-\beta|^{\s}}  (C''_h)^{|\beta|^{\s}}  |\alpha-\beta|^{\t |\alpha-\beta|^{\s}}|\beta|^{\t |\beta|^{\s}}}{ |\beta|^{(2^{\s}-1)\t |\beta|^{\s}}(2h)^{|\alpha-\beta|} |\alpha-\beta|^{2^{\s-1}\t |\alpha - \beta|^{\s}}} \\
\leq \frac{(C''_h h)^{|\beta|^{\s}}}{|\beta|^{\t(2^{\s}-2)|\beta|^{\s}}} C \sum\limits_{\delta\in \N^d}\frac{(C''_h)^{|\delta|^{\s}}  }{ (2h)^{|\delta|} |\delta|^{(2^{\s-1}-1)\t |\delta|^{\s}}} = C'''_h \frac{(C''_h h)^{|\beta|^{\s}}}{|\beta|^{\t(2^{\s}-2)|\beta|^{\s}}}.\nonumber
\end{multline} In particular, $\sup_{\beta\in \N^d}I_{2,\beta}\leq C'''_h \exp\{T_{(2^{\s}-2)\t, \s, C''_h h}(1)\}<\infty$.

Now,  Stoke's formula  gives
\begin{multline}
\label{StoksovaFormula}
\langle F(x+i 0),\varphi (x)\rangle= \int_K F(x+iY)\Phi (x+iY)dx\\
+2i \sum_{j=1}^d Y_j \int_{0}^1 \int_K \partial_{\overline{z_j}}\Phi(x+itY)F(x+itY)dt dx,
\end{multline}
and we have used the assumptions in Theorem \ref{GlavnaTeorema}, and inequalities \eqref{prvaocenaPhi} and \eqref{AnaliticExt}.

Note that for $H=h$, \eqref{uslovVelikoF} and \eqref{prvaocenaPhi} imply that there exists $A_h>0$ such that
\begin{multline}
\label{PrviIntegral}
|  F(x+iY)\Phi (x+iY) |\leq A_h \|\varphi\|_{\E_{\t/2^{\s-1},\s,h}} \exp\{T_{(2^{\s}-1)\t,\s,h}(1/|Y|)\} \\
= A'_h \|\varphi\|_{\E_{\t/2^{\s-1},\s,h}}, \quad x\in K,
\end{multline}
where $A'_h= A_h\, \exp\{T_{(2^{\s}-1)\t,\s,h}(1/|Y|)\}$.

Moreover,  \eqref{uslovVelikoF} and \eqref{AnaliticExt}  imply that there exists $B_h>0$ such that
\begin{equation}
\label{DrugiIntegral}
|\partial_{\overline{z_j}}\Phi(z)F(z)| \leq B_h \|\varphi\|_{\E_{\t/2^{\s-1},\s,h}},\quad 1\leq j\leq d,\,z\in Z_Y.
\end{equation} Now \eqref{StoksovaFormula}, \eqref{PrviIntegral} and \eqref{DrugiIntegral} implies
$$|\langle F(x+i 0),\varphi (x)\rangle|\leq B'_h \|\varphi\|_{\E_{\t/2^{\s-1},\s,h}}, $$ for suitable constant $B'_h>0$. This completes the proof of the second part of theorem, and the first part follows immediately.

\end{proof}

\begin{rem} \label{remark-structure}
In  order to show that any ultradistribution $f$ is locally (on a bounded open set $U$) the sum of boundary values of analytic functions defined in the corresponding cone domains $\Gamma_j$, $j=1,...,k$, one should proceed as in the classical theory. We multiply $f$ by a cutoff test function $\kappa_U$ equal to $1$  over $U$, and obtain $f_0=f\kappa_U$ equals $f$ on $U.$  Then we divide $\mathbb R^n\setminus\{0\}$ into regular non overlapping cones $\Gamma_{j0}$, $j=1,...,k$, dual cones of $\Gamma_j$, and define
$$
F_j(z)=\langle f_0(t),\int_{\Gamma_{j0}} \exp\{2\pi i(z-t)\eta\}d\eta\rangle, \quad z\in \mathbb R^n + i \Gamma_j, \;\; j=1,...,k.
$$
Now one can get the growth conditions for $F_j, j=1,...,k,$ and show that
$$f_0=\sum_{j=1}^k\lim_{y\rightarrow 0,y\in \Gamma_j}F_j(x+iy),\quad x\in\mathbb R^n.$$
The details  will be given in a separate contribution where we will consider $L^p$ versions of new ultradistributions spaces similar to the corresponding ones   in \cite{PilipovicKnjiga}.
\end{rem}

\section{Wave front sets}\label{secTalasniFront}

In this section we analyze wave front sets $\WF_{\t,\s}(u)$ related to the classes $\E_{\t,\s}(U)$ introduced in Section \ref{secSpaces}. We refer to  \cite{ PTT-02, PTT-03, PTT-04, TT0, TT} for properties of $\WF_{\t,\s}(u)$ when $u$ is a Schwartz distribution.

We begin with the  definition.
\begin{de}
\label{WFtausigma}
Let $\t>0$, $\s>1$, $U$ open set in $\Rd$ and $(x_0,\xi_0)\in U\times \Rd\backslash\{0\}$. Then for $u\in \D'_{\{\t,\s\}}(U)$ (respectively $ \D'_{(\t,\s)}(U)$),  $(x_0,\xi_0)\not \in \WF_{\{\t,\s\}}(u)$ (resp. $(x_0,\xi_0)\not \in \WF_{(\t,\s)}(u)$) if and only if there exists a conic
neighborhood $\Gamma$ of $\xi_0$, a compact neighborhood
$ K $ of $x_0$, and
$\phi\in \D_{\{\t,\s\}}(U)$ (respectively $\phi\in \D_{(\t,\s)}(U)$)  such that $\supp\phi\subseteq K$, $\phi=1$ on some neighborhood of $x_0$, and
\be
\label{WFuslov1}
|\widehat{\phi u}(\xi)|\leq A \exp\{-T_{\t,\s,h}(|\xi|)\},\quad \xi\in \Gamma\,,\nonumber
\ee
for some $A,h>0$ (resp. for any $h>0$ there exists $A>0$).
\end{de}

We will write $\WF_{\t,\s}(u)$ for $\WF_{(\t,\s)}(u)$ or $\WF_{\{\t,\s\}}(u)$.
\begin{rem}
Note that  $\WF_{\t,1}(u)=\WF_{\t}(u)$, $\t>1$, are Gevrey wave font sets investigated in \cite{Rodino}.
\end{rem}

Moreover (cf. \cite{PTT-03}), for $0<\t_1<\t_2$ and $\s>1$ we have
$$
\WF (u) \subset \WF_{\t_2,\s}(u)\subset \WF_{\t_1,\s}(u)\subset \bigcap_{t>1}\WF_{t}(u)\subset \WF_A (u),\quad u\in \D'(U),
$$
where  $\WF_A$ denotes the analytic wave front set.

Let
\be
\WF^{\{\s\}}(u)=\bigcap_{\t>0} \WF_{\t,\s} (u), \quad u\in  {\D}^{'\{\s\}}(U), \nonumber
\ee
and
\be
\WF^{(\s)}(u)=\bigcup_{\t>0} \WF_{\t,\s} (u), \quad u\in  {\D}^{'(\s)}(U).\nonumber
\ee
For such wave front sets we have the following corollary which is an immediate consequence of Lemma \ref{OcenazaAsociranu}.
\begin{cor}
\label{WFSigma}
Let $u\in \D^{'\{\s\}}(U)$ (resp. $ \D^{'(\s)}(U)$), $\s>1$. Then $(x_0,\xi_0)\not \in \WF^{\{\s\}}(u)$ (resp. $(x_0,\xi_0)\not \in \WF^{(\s)}(u)$) if and only if there exists a conic
neighborhood $\Gamma$ of $\xi_0$, a compact neighborhood
$ K $ of $x_0$, and
$\phi\in \D^{\{\s\}}(U)$ (resp $\phi\in \D^{(\s)}(U)$)  such that $\supp\phi\subseteq K$, $\phi=1$ on some neighborhood of $x_0$, and
\be
\label{WFuslov2}
|\widehat{\phi u}(\xi)|\leq  A |\xi|^{ - H \Big( \frac{\ln |\xi|}{\ln ( \ln  |\xi|) }\Big)^{\frac{1}{\s-1}}},\quad \xi\in \Gamma\,,\nonumber
\ee
for some $A,H>0$ (resp. for any $H>0$ there exists $A>0$).
\end{cor}


We write $u(x)=F(x+i\Gamma \,0)$ if $u(x)$ is obtained as boundary value of an analytic function $F(x+iy)$ as $y\to 0$ in $\Gamma$. Recall (cf. \cite{H})
$$\Gamma_0=\{\xi \in \Rd\,|\, y\cdot \xi \geq 0\,{\rm{\,for\,\, all\,}}\,y\in \Gamma\}$$ denotes the dual cone of $\Gamma$.

To conclude the paper we prove the following theorem.

\begin{te}
\label{GlavnaTeorema2}
Let the assumptions of Theorem \ref{GlavnaTeorema} hold, and let $u(x)=F(x+i\Gamma \,0)\in  \D^{'(\s)}(U)$ (resp. $\D^{'\{\s\}}(U)$). Then
\be
\WF^{(\s)}(u)\subseteq U\times \Gamma_0,\quad(resp.\,\,\WF^{\{\s\}}(u)\subseteq U \times \Gamma_0).\nonumber
\ee
More precisely, if $u(x)=F(x+i\Gamma \,0)\in  \D'_{\{\t/2^{\s-1},\s\}}(U)$ (resp. $\D'_{(\t/2^{\s-1},\s)}(U)$) then
\be
\label{Posledica2}
\WF_{\{(2^{\s}-1)\t,\s\}}(u)\subseteq U\times \Gamma_0,\quad(resp.\,\,\ \WF_{((2^{\s}-1)\t,\s)}\subseteq U \times \Gamma_0).\nonumber
\ee
\end{te}

\begin{proof}

Fix $x_0\in U$ and $\xi_0\not\in \Gamma_0\backslash\{0\}$. Then there exist  $ Y=(Y_1,\dots, Y_d)\in \Gamma$, $|Y|<\gamma$, such that $Y\cdot \xi_0<0$. Moreover, there exists conical neighborhood $V$ of $\xi_0$ and constant $\gamma_1>0$ such that $Y\cdot \xi\leq-\gamma_1 |\xi|$, for all $\xi\in V$. To see that, note that there exists $B_{r}(\xi_0)$ such that $Y\cdot \xi<0$ for all $\xi\in B_{r}(\xi_0)$. The assertion follows for $\dss V=\{s\xi \,|\, s>0,\, \xi\in B_{r}(\xi_0) \}$ and $\dss \gamma_1=\inf_{\xi\in V,\,|\xi|=1}(-Y)\cdot \xi$.

Let $\t>0$ and $\t_0 = (2^{\s}-1) \t$.  If $u(x)=F(x+i\Gamma \,0)\in  \D'_{\t/2^{\s-1},\s}(U)$ as in Theorem \ref{GlavnaTeorema}, then

\be
\label{uslovVelikoFPosledica1}
|F(z)|\leq A\, \exp\{T_{\t_0,\s,h_1}(1/|y|)\},\quad z=x+iy \in Z,
\ee
for suitable constants $ A, h_1>0$.

Choose $\varphi\in\D^K_{\t/4^{\s-1},\s}$ such that $\varphi=1$ in a neighborhood of $x_0$ and let $Z_Y$ be as in \eqref{ZY}. Then there exists $\Phi$ (see \eqref{AlmostAnalyticExtension}) such that
\be
\label{AnaliticExtPosledica}
|\Phi(z)|\leq A_1, \;\;\; \text{and} \;\;\; |\partial_{\overline{z_j}}\Phi(z)|\leq A_2 \exp\{-T_{\t_0/2^{\s-1},\s,h_2}(1/|tY|)\},
\ee
$ z\in Z_Y$, $1\leq j \leq d$, for suitable constants $A_1,A_2,h_2>0$.

Note that  formula \eqref{StoksovaFormula} implies

\begin{multline}
\label{StoksovaFormulaWF}
\widehat{(\varphi u)}(\xi)=\langle u(x) e^{-i x\cdot \xi},\varphi (x)\rangle= \int_K F(x+iY) e^{-i (x+i Y)\cdot\xi}\Phi (x+iY)dx\\
+2i \sum_{j=1}^d Y_j \int_{0}^1 \int_K\partial_{\overline{z_j}}\Phi(x+itY)F(x+itY)e^{-i (x+i tY)\cdot\xi}dt dx, \quad\xi \in V.
\end{multline} Using \eqref{uslovVelikoFPosledica1} and \eqref{AnaliticExtPosledica} we have
\be
\label{StoksovaFormulaWF1}
|F(x+iY) \Phi(x+iY)e^{-i(x+iY)\xi}|\leq B\,  e^{-\gamma_1 |\xi|},\quad x\in K,\, \xi\in V,
\ee for some $B>0$.

Moreover, for $z\in Z_Y$ and $\xi\in V$ we have
\begin{multline}
\label{StoksovaFormulaWF2}
|F(z) \partial_{\overline{z_j}}\Phi(z)e^{-iz\cdot\xi}| \\
\leq  C \exp\{T_{\t_0,\s,h_1}(1/|tY|)-T_{\t_0/2^{\s-1},\s,h_2}(1/|tY|)-t\gamma_1 |\xi|\}\\
\leq  C_1 \exp\{{-T_{\t_0,\s,c_{h_1,h_2}}(1/(t\gamma_1))-t\gamma_1|\xi|}\}
\leq C_ 2 \exp\{{-T_{\t_0,\s,c'_{h_1,h_2}}(|\xi|)}\},
\end{multline}
for suitable constants $C_1,C_2, c_{h_1,h_2}, c'_{h_1,h_2}>0$, where we have used inequalities \eqref{Submulti} and \eqref{Submulti1}.

Finally, using \eqref{StoksovaFormulaWF}, \eqref{StoksovaFormulaWF1} and \eqref{StoksovaFormulaWF2} we obtain
\be
|\widehat{(\varphi u)}(\xi)|\leq B_1 (e^{-\gamma_1 |\xi|}+\exp\{-T_{\t_0,\s,c'_{h_1,h_2}}(|\xi|) \} )\leq B_ 2 \exp\{-T_{\t_0,\s,c'_{h_1,h_2}}(|\xi|)\},\nonumber
\ee
for $\xi\in V$ and for suitable constant  $B_2>0$. This completes the proof.
\end{proof}

\section*{Appendix}

In the following Lemma we study  $T_{\t,\s,h}(k)$ in some details.

\begin{lema}
\label{NejednakostLema2}
Let $h>0$, and $T_{\t,\s,h}$ be given by \eqref{asociranaProduzena}, and let $T^{*}_{\t,\s,h} $ be given by \eqref{Tstar}. Then
\begin{itemize}
\item[$a)$] if $h_1<h_2$ then $\dss {T_{\t,\s,h_1}(k)}<{T_{\t,\s,h_2}(k)}$, $k>0$. Moreover, for any $h>0$ there exists $H>h$ such that
\be
\label{Rastpoh}
{T_{\t,\s,h}(k)}\leq {T^{*}_{\t,\s,h}(k)}\leq {T_{\t,\s,H}(k)},\quad k>0.
\ee

\item[$b)$] for $h_1,h_2>0$ there exists $C,c_{h_1,h_2}>0$ such that
\be
\label{Submulti}
{T_{\t,\s,h_1} (k)} + {T_{\t,\s,h_2}(k)} \leq{T_{\t/2^{\s-1},\s,c_{h_1,h_2}}(k)} +\ln C
\quad k>0,
\ee
\item [$c)$] for every $h>0$ there exits $H>0$ such that

\be
\label{Submulti1}
{T_{\t,\s,H}(l)}  \leq {T_{\t,\s,h}(1/k)+k l} ,\quad k,l>0.
\ee
\end{itemize} 

\end{lema}
\begin{proof}
$a)$ Notice that for arbitrary $h>0$,

$$\ln\frac{h^{p^{\s}}k^{p}}{p^{\t p^{\s}}}\leq\ln \frac{p^{ p}h^{p^{\s}}k^{p}}{p^{\t p^{\s}}}\leq \ln\frac{(C h)^{p^{\s}}k^{p}}{p^{\t p^{\s}}},\quad k>0,$$ where for the second inequality we use that for every $\s>1$ there exists $C>1$ such that $p^{ p}\leq C^{p^\s}$, $ p\in \N$ (see the proof of Proposition 2.1. in \cite{PTT-01}). Now
\eqref{Rastpoh} follows by putting $H=Ch$.

$ b)$ Let $h_1,h_2>0$. We will use the following simple inequality
\be
\label{SimpleInequality}
\dss p^{\s}+q^{\s} \leq (p+q)^{\s}\leq 2^{\s-1}(p^{\s}+q^{\s}), \quad p,q\in \N.
\ee Since, $\dss h_1^{p^{\s}}h_2^{q^{\s}}\leq (h_1+h_2)^{p^{\s}+q^{\s}}$ we conclude that $ h_1^{p^{\s}}h_2^{q^{\s}}\leq (h_1+h_2)^{(p+q)^{\s}}$ when $h_1+h_2 \geq 1$ and  $\dss  h_1^{p^{\s}}h_2^{q^{\s}}\leq (h_1+h_2)^{(1/2^{\s-1})(p+q)^{\s}}$ when $0<h_1+h_2<1$.

Hence there exists $0<c_{\s}\leq 1$ such that
$$\ln\frac{{h_1}^{p^{\s}} k^p}{p^{\t p^{\s}}}+\ln \frac{{h_2}^{q^{\s}} k^q}{q^{\t q^{\s}}}\leq\ln \frac{(C(h_1+h_2)^{c_{\s}})^{(p+q)^{\s}} k^{p+q}}{(p+q)^{(\t /2^{\s-1})(p+q)^{\s}}}+\ln C,\quad p,q\in \N,$$ where $C>0$ is constant appearing in $\widetilde{(M.2)}$. Now \eqref{Submulti} follows after taking supremums over $p,q\in \N$.

$d)$ Recall (see \cite{KomatsuNotes}), there exists $A>0$ such that $\dss {k l}=\sup_{p\in\N}\ln\frac {A^p k^p l^p}{p^p}$.
Note that for every $\s>1$ there exists $0<C<1$ such that $\dss \frac{1}{p^p}\geq C^{p^{\s}},\quad p\in \N.$

Then for arbitrary $h>0$ we have
\begin{multline}
{T_{\t,\s,h}(1/k)+k l}=\sup_{p,q\in \N}\ln \frac{h^{p^{\s}}}{k^p p^{\t p^{\s}}}\frac{A^q k^q l^q}{q^q}\geq \sup_{p,q\in \N, p=q}\ln \frac{(A'Ch)^{p^{\s}}l ^{p}}{p^{\t p^{\s}}}\\=T_{\t,\s,H}( l),\quad k,l>0,\nonumber
\end{multline} where $A' = \min\{1, A\}$. This proves \eqref{Submulti1}.
\end{proof}

Finally we discuss certain stability and embedding properties of ${\E}_{\t, \s}(U)$ given by \eqref{NewClassesInd} and \eqref{NewClassesProj}. Analogous considerations hold when the spaces $\D_{\t,\s}(U)$ from \eqref{NewClassesInd2} and \eqref{NewClassesProj2}
are considered instead.

Let  $a_{\alpha} \in {\E}_{(\t, \s)}(U)$ (resp. $ a_{\alpha} \in {\E}_{\{\t, \s\}}(U)$), where $U$ is an open set in
$\Rd$. Then we say that
$$
P(x,\partial)=\suml_{|\alpha|=0}^{\infty}a_{\alpha}(x){\partial}^{\alpha}
$$
is an ultradifferentiable operators of class $(\t,\s)$ (resp. $\{\t,\s\}$) on  $U$ if for every $K\subset\subset U$ there exists constant $L>0$ such that for any $h>0$ there exists $A>0$ (resp. for every $K\subset\subset U$ there exists $h>0$ such that for any $L>0$ there exists $A>0$) such that,
\begin{equation*}
\label{Operatortausigma}
\sup_{x\in K}|\partial^{\beta}a_{\alpha}(x)|\leq A h^{{|\beta|}^{\s}}|\beta|^{\t{|\beta|}^{\s}}\frac{L^{|\alpha|^{\s}}}{|\alpha|^{\t 2^{\s-1}{|\alpha|}^{\s}}},\quad {\alpha,\beta \in \N^d}.
\end{equation*}

\par

We refer to \cite{PTT-02} for the proof of the following continuity and embedding properties.

\begin{prop}
\label{TeoremaZatvorenostUltraDfOP}
\begin{itemize}\item [$a)$]  Let $P(x,\partial)$ be a differential operator of class $(\t,\s)$ (resp. $\{\t,\s\}$). Then $\dss P(x,\partial):\quad {\E}_{\t, \s}(U) \longrightarrow {\E}_{\t 2^{\s-1}, \s}(U)$ is a continuous linear mapping; the same holds for
$$
P(x,\partial):\quad \varinjlim_{\t\to \infty}\E_{\t,\s}(U) \longrightarrow \varinjlim_{\t\to \infty}\E_{\t,\s}(U).
$$

\item [$b)$]  Let $\s_1\geq 1$. Then for every $\s_2>\s_1$
\begin{equation*} \label{Theta_S_embedd}
\varinjlim_{\t\to \infty}{\E}_{\t,
{\s_1}}(U)\hookrightarrow \varprojlim_{\t\to 0^+} {\E}_{\t,
{\s_2}}(U).
\end{equation*}
\item [$c)$] If $0<\t_1<\t_2$, then
\be \label{RoumieuBeurling} \E_{\{\t_1,\s\}}(U)\hookrightarrow
\E_{(\t_2,\s)}(U)\hookrightarrow \E_{\{\t_2,\s\}}(U), \;\;\; \s> 1,\nonumber \ee
and
$$
\varinjlim_{\t\to \infty}{\E}_{\{\t, {\s}\}}(U)= \varinjlim_{\t\to \infty} {\E}_{(\t, {\s})}(U),
$$
$$
\varprojlim_{\t\to 0^+}{\E}_{\{\t, {\s}\}}(U)= \varprojlim_{\t\to 0^+} {\E}_{(\t, {\s})}(U), \;\;\; \s> 1.
$$
\end{itemize} Consequently we obtain that

\begin{equation*}
\label{GevreyNewclass}
\varinjlim_{t\to\infty} \E_{t}(U)\hookrightarrow {\E}_{\tau, \s}(U)
\hookrightarrow  C^{\infty}(U), \;\;\;
\tau>0, \; \s>1,
\end{equation*} where $ \E_{t}(U)$ is Gevrey space with index $t>1$.
\end{prop}

\section*{Acknowledgement}
This research has been supported by the Ministry of Education, Science and
Technological Development through the Projects no. 451-03-68/2020-14/200125 and { 451-03-68/2020-14/200156},
and Project 19.032/961103/19 MNRVOID of the Republic of Srpska.

\end{document}